\documentclass{article}
\usepackage{amsmath, amssymb, amsthm, relsize, tikz-cd}

\theoremstyle{plain}
\newtheorem{theorem}{Theorem}[section]

\theoremstyle{definition}

\DeclareMathOperator*{\OmSum}{\mathlarger{\mathlarger{\Omega}}}

\newcommand{\tet}{\text{tet}_\beta}

\begin{document}

\title{Infinite Compositions and Complex Dynamics; Generalizing Schr\"{o}der and Abel Functions}

\author{James David Nixon\\
	JmsNxn92@gmail.com\\}

\maketitle

\begin{abstract}
Using infinite compositions, we solve the general equations $P(\lambda w) = p(w)f(P(w))$ for holomorphic functions $p$ and $f$. We describe the situations in which this equation is palpable; and their effectiveness at describing dynamical properties of the orbit $f^{\circ n}(z)$. We similarly make a change of variables to study a generalized form of the Abel equation, $F(s+1) = u(s)f(F(s))$. This paper is intended as a more in depth examination of work done previously in \cite{Nix5}.
\end{abstract}

\emph{Keywords:} Complex Analysis; Infinite Compositions; Complex Dynamics.\\

\emph{2010 Mathematics Subject Classification:} 30D05; 30B50; 37F10; 39B12; 39B32\\

\section{Introduction}\label{sec1}
\setcounter{equation}{0}

This paper will be a tad expansive; in that, we want to cover some of the work of John Milnor's book, \textit{Dynamics In One Complex Variable.} Except, we want to generalize the idea of the Fatou coordinate, and the Schr\"{o}der function. And, we only have a bit of space. This paper is intended as supplemental reading for our recent paper, \textit{The Limits of a Family; of Asymptotic Solutions to the Tetration Equation} \cite{Nix5}.

We want to envelop the same schema that Milnor uses in classifying complex dynamics. Whereupon, where Milnor focuses on Local Fixed Point Theory, we'll avoid the Fixed Point part, and only talk about a local theory. And where Milnor focuses on a global theory; we again will focus on varying equations than the coordinate changes Milnor describes.

To set the stage, we recall the three types of dynamics Milnor chooses to organize his study. These being,

\begin{eqnarray*}
f&:& \mathbb{D} \to \mathbb{D}\,\,\text{for}\,\, \mathbb{D} = \{z\in\mathbb{C}\,|\,|z|<1\}\\
f &:& \mathbb{C} \to \mathbb{C}\\
f &:& \widehat{\mathbb{C}} \to \widehat{\mathbb{C}}\,\,\text{for}\,\,\widehat{\mathbb{C}} = \mathbb{C} \cup \{\infty\}\\
\end{eqnarray*}

Where we assume all of these maps are holomorphic. To the trained mathematician, I have somewhat butchered an otherwise far more diverse setting; but this is the reduction of the three cases we'll focus on. Which is to say; for our purposes this rendition is well enough. In the first case, it results in a local fixed point theory. In the second case, it results in the global dynamics of an entire function. And the third case equates to rational maps on the Riemann Sphere.

In each of these spaces, we want to construct coordinate changes, which have a slightly different structure than Fatou coordinates, or Koenig's coordinate. Where in these cases,

\[
\alpha(f(z)) = f(z)+1\\
\]

Or,

\[
\Psi(f(z)) = \lambda \Psi(z)\\
\]

For appropriate functions; and these things are appropriately unique. Milnor describes a very complex renaissance painting of a how to do complex dynamics--where these 2 equations pop up consistently. 

In this paper, we're going to focus on slightly different equations; which have much of the same structure. However, these equations will be more esoteric. They will invariably require a dependence on $z$, other than $f$'s and $\Psi$'s and $\alpha$'s. They'll depend on $z$ like a multiplicative factor.

The benefit of doing such, is that there are far more exotic scenarios where these equations can be constructed. And that, the strongest manner of constructing these coordinate changes, is to rely on iterative behaviour near a fixed point. In the world the author is proposing, we don't need to do that. In the world of Infinite Compositions, this is supplemental.\\

The author again argues for the benefit of Infinite Compositions in the study of Functional Equations in the Complex Plane. He cites \cite{Nix,Nix2,Nix3,Nix4,Nix5}, as works he's already spent on infinite compositions and their manner of convergence. Most recently in \cite{Nix5} he implicitly used much of the theory he will use in this paper. Where that paper used the theory of infinite compositions to construct an Abel function of the exponential map $e^z$ that was real-valued. 

The author can no longer press himself to include an introduction to infinite compositions in every paper he writes. For that reason we assume the reader is familiar with $\OmSum$-notation. We no longer introduce it, and much of this paper will glaze over manipulations using $\OmSum$-notation.

\section{The Local Theory}

The local theory of complex dynamics is perhaps the easiest to translate to our scenario. It is the scenario where we get to see the avoidance of much of traditional complex dynamics, while still dealing with complex dynamics. That is, we don't need normality theorems on the orbits; though we do need normality theorems. It's just, these normality theorems were all derived in previous works, and have nothing to do with the Julia set or the Fatou set.

We may consider the local theory, as a place where $f$ (our function of interest), sends a domain to itself. This accounts for either a component of the Fatou set which maps to itself; or the Julia set (which maps to itself, but usually has trivial interior); or some combination of these ideas. Milnor treats the local theory as about a fixed point; we again, don't care about fixed points (so much as we care of a domain where $f$ sends to itself (where fixed points give us a good language at determining these domains)).

But I digress.

Let's take a function $f:\mathbb{D} \to \mathbb{D}$; and call it our function of interest. Let's take another function $p(w) : \mathbb{D} \to \mathbb{D}$ with $p(0) = 0$. Now, we don't care whether $f$ has fixed points, how many it has, whether it has a periodic cycle; we only care it maps $\mathbb{D} \to \mathbb{D}$. From these two functions, we want to construct a function $P_\lambda : \mathbb{D} \to \mathbb{D}$ such that; for $0 < |\lambda| < 1$,

\[
P_\lambda(\lambda^{-1} w) = p(w)f(P_\lambda(w))\\
\]

Finding this function is surprisingly simple. And further, showing holomorphy in $z$ and $\lambda$, is also surprisingly simple. We point the reader towards \cite{Nix2,Nix3,Nix5}, where this theorem was shown on each occasion. The general form of the theorem, we'll write below,

\begin{theorem}\label{thmA}
Let $\{H_j(s,z)\}_{j=1}^\infty$ be a sequence of holomorphic functions such that $H_j(s,z) : \mathcal{S} \times \mathcal{G} \to \mathcal{G}$ where $\mathcal{S}$ and $\mathcal{G}$ are domains in $\mathbb{C}$. Suppose there exists some $A \in \mathcal{G}$, such for all compact sets $\mathcal{N}\subset\mathcal{G}$, the following sum converges,

\[
\sum_{j=1}^\infty ||H_j(s,z) - A||_{z \in \mathcal{N},s \in \mathcal{S}} = \sum_{j=1}^\infty \sup_{z \in \mathcal{N},s \in \mathcal{S}}|H_j(s,z) - A| < \infty
\]

Then the expression,

\[
H(s) = \lim_{n\to\infty}\OmSum_{j=1}^n H_j(s,z)\bullet z = \lim_{n\to\infty} H_1(s,H_2(s,...H_n(s,z)))\\
\]

Converges uniformly for $s \in \mathcal{S}$ and $z \in \mathcal{N}$ as $n\to\infty$ to $H$, a holomorphic function in $s\in\mathcal{S}$, constant in $z$.
\end{theorem}

And we use this theorem as a black box to construct $P_\lambda$. Let us call $\mathbb{D}_\rho = \{z \in \mathbb{D}\,|\,|z| \le \rho\}$ for $0 < \rho <1$. Then, for all $0 < \rho < 1$ the following sum converges,

\[
\sum_{j=1}^\infty ||p(\lambda^j w)f(z)||_{(\lambda,w,z)\in\mathbb{D}_\rho^3} < \infty\\
\]

This implies, by Theorem \ref{thmA}, the infinite composition,

\[
P_\lambda(w) = \OmSum_{j=1}^\infty p(\lambda^jw)f(z)\bullet z\\
\]

Converges to a holomorphic function for $|w|,|\lambda|<1$. And this function has the identity,

\begin{eqnarray*}
P_\lambda(\lambda^{-1} w) &=& \OmSum_{j=1}^\infty p(\lambda^j \lambda^{-1}w)f(z)\bullet z\\
&=& \OmSum_{j=1}^\infty p(\lambda^{j-1}w)f(z)\bullet z\\
&=& \OmSum_{j=0}^\infty p(\lambda^j w)f(z)\bullet z\\
&=& p(w)f(z) \bullet \OmSum_{j=1}^\infty p(\lambda^j w)f(z)\bullet z\\
&=& p(w)f(P_\lambda(w))\\
\end{eqnarray*}

And from this, we've constructed a mock Schr\"{o}der coordinate. We write this as a theorem.

\begin{theorem}[The Modified Schr\"{o}der Theorem]\label{thmSCH}
Let $f:\mathbb{D} \to \mathbb{D}$ be a holomorphic function. For an arbitrary holomorphic function $p:\mathbb{D} \to \mathbb{D}$ with $p(0) = 0$ there exists a holomorphic function $P_\lambda(w)$ for $\lambda,w \in \mathbb{D}$ such that,

\[
P_\lambda(\frac{w}{\lambda}) =p(w)f(P_\lambda(w))\\
\]

Where,

\[
P_\lambda(w) = \OmSum_{j=1}^\infty p(\lambda^jw)f(z)\,\bullet z = \lim_{n\to\infty} p(\lambda w)f(p(\lambda^2 w)f(...p(\lambda^nw)f(z)))\\
\]
\end{theorem}

Where it is important now to put this in context to the local fixed point theory. If $f$ has an attracting fixed point at $0$, we can always construct a function $P$. And the extent to which we can construct $P$, depends entirely on where $f$ is defined, and where the sum,

\[
\sum_{j=1}^\infty p(\lambda^j w) f(z)\\
\]

Converges, which is possibly all of $\mathbb{C}$. However, we don't need there to be a fixed point on $\mathbb{D}$; if we were to switch this to an arbitrary simply connected domain; it's only required it maps the domain to itself.

Where here, the Julia set will not affect anything, nor approaching a different basin of attraction. The function $P$ will exist independently of this. So long as the various modified orbits $\OmSum_{j=1}^n p(\lambda^j w) f(z) \bullet z$ are defined (we do not leave the domain of holomorphy), everything is golden.

Constructing a Taylor series of $P_\lambda(w)$ in $w$ is also fairly routine. If $f$ and $p$ have well explained Taylor series, this proposes an alternative avenue. We simply take,

\[
P'_\lambda(\frac{w}{\lambda})/\lambda = p'(w)f(P_\lambda(w)) + p(w)f'(P_\lambda(w)) P'_\lambda(w)\\
\]

And iterate this relationship when taking derivatives to get our Taylor series in $w$ in a neighborhood of $0$; where we recall $p(0) = 0$ and $P(0) = 0$. We can also express a formula for the derivative using infinite compositions. This is a tad more exhausting, but doable.

\begin{eqnarray*}
P'_\lambda(w) &=& \OmSum_{j=1}^\infty \left(\lambda^j p'(\lambda^j w)f(P_\lambda(\lambda^j w)) + \lambda^jp(\lambda^jw)f'(P_\lambda(\lambda^j w)) z\right)\,\bullet z\\
\end{eqnarray*}

Which is an inductive process. This is a difficult idea to explain. But, we are solving a functional equation using infinite compositions; and the derivative of that solution satisfies a linear functional equation. And this process continues for further derivatives. As we don't really need this; it is left to the reader.

The Taylor Series method is much more efficient, and the intractable nature of infinite compositions (especially to the newly initiated), pushes the author away from this language. Though, both are rather equivalent, and we are just solving a recurrence relation in the Taylor series; though hidden beneath the $\OmSum$-notation. For this reason further, he will not broach this avenue too much.\\

The second way to view the local theory, is to solve the Abel equation. This is done almost offhandedly. We make the change of variables $w = \lambda^{-s}$, then our new function $F_\lambda(s)$ is holomorphic for $\Re(s) < 0$, and satisfies the equation,

\[
F_\lambda(s+1) = u(s)f(F_\lambda(s))\\
\]

Where $u(s) = p(\lambda^{-s})$. In this way, we can think of these solutions more intuitively. They are very much functions that are nearly fractional iterates of our function $f$. Wherein, it is nearly a Fatou coordinate, minus a multiplicative factor (which I call a convergent), which keeps all of our messes in-line.

As simple as the Local theory is, is also as unenlightening it is. The real power of this method lies in entire function theory--where appropriate choices of $p$ can be quite useful for discussing asymptotics at infinity. We move on to entire functions, then.

\section{Entire functions and the mock Abel equation}\label{sec3}
\setcounter{equation}{0}

In this section we'll focus on functions $f:\mathbb{C} \to \mathbb{C}$, where here the mock Abel equation is king. It also helps, in this case, to keep another well rounded space in our back pocket. Call $\mathfrak{E}$ the space of functions $u : \mathcal{G} \to \mathcal{G}$ such that,

\[
\sum_{j=1}^\infty ||u(s-j)||_{s\in\mathcal{S}} < \infty\\
\]

For all compact sets $\mathcal{S} \subset \mathcal{G}$. The domain $\mathcal{G}$ is nearly arbitrary, so long as left shifts are within it; as such the point at $-\infty$ is included as an accumulation point on the Riemann sphere. It is apt to think of $\mathcal{G}$ as a horizontal strip in $\mathbb{C}$ or something that looks like this. Then, we'll start with a theorem for our discussion.

\begin{theorem}[The Modified Abel Theorem]\label{thmABL}
Let $f:\mathbb{C} \to \mathbb{C}$ be an entire function. Let $\mathcal{G} \subseteq \mathbb{C}$ be a domain in which $s \in \mathcal{G} \Rightarrow s-1 \in \mathcal{G}$. Suppose $u : \mathcal{G} \to \mathbb{C}$ and for all compact sets $\mathcal{S} \subset \mathcal{G}$,

\[
\sum_{j=1}^\infty ||u(s-j)||_{s\in\mathcal{S}} < \infty\\
\]

Then there exists a function $F:\mathcal{G} \to \mathbb{C}$ such that,

\[
F(s+1) = u(s)f(F(s))\\
\]

Where,

\[
F(s) = \OmSum_{j=1}^\infty u(s-j)f(z)\,\bullet z\\
\]
\end{theorem}

\begin{proof}
This will entirely follow from Theorem \ref{thmA}. For all compact sets $\mathcal{S} \subset \mathcal{G}$ and all compact sets $\mathcal{K}\subset\mathbb{C}$ we have that,

\[
\sum_{j=1}^\infty ||u(s-j)f(z)||_{s\in\mathcal{S}, z \in \mathcal{K}}  <\infty\\
\]

Which concludes the proof.
\end{proof}

The author would like to enter two functions in which this process works--and is elegant. The first version we'll take is the exponential version. For this, we just set $u=e^s$. And we'll call the function,

\[
F_{\exp}(s) = \OmSum_{j=1}^\infty e^{s-j}f(z)\,\bullet z\\
\]

Which satisfies,

\[
F_{\exp}(s+1) = e^s f(F_{\exp}(s))\\
\]

These functions were used gratuitously in \cite{Nix2}--but not in a holomorphic setting. They were used to construct a smooth sequence of functions $\{h_n\}_{n=1}^\infty$ such that $h_{n+1}(x+1) = h_n(h_{n+1}(x))$ with the initial condition $h_1(x) = e^x$. 

Using the exponential function is very beneficial for very fast growing functions, as it over estimates the dynamics. This is to mean, that if we were to solve the Abel equation $f(G(s)) = G(s+1)$, then our function $F_{\exp}$ will inherently grow much faster--there will be an exponential component attached to it. 

You might ask, why would we want to over-estimate? The reason we did it, was so that the inverse operation $f^{-1}$ is greater than $x$. Which would mean $f^{-1}(F_{\exp}(x+1)) > F_{\exp}(x)$. This was used in a very convenient manner in \cite{Nix2}--which can't really be summarized here. But it constructed smooth hyper-operators.\\

The second type of function $u$ is a tad more beneficial in complex dynamics. This function serves to describe the dynamics of the function $f$ at $\infty$. We will write a quick theorem here, which highlights this. But first, we'll introduce it. Consider the family of logistic maps (sort of),

\[
u_\beta(s) = \frac{1}{e^{-\beta s} + 1}\\
\]

Where here, $\Re \beta > 0$. This function is holomorphic for $\beta s \neq (2 k +1)\pi i$ for $k\in \mathbb{Z}$. Then, we can develop the function,

\[
F_\beta(s) = \OmSum_{j=1}^\infty \frac{f(z)}{e^{\beta(j-s)} + 1}\,\bullet z\\
\]

Which is holomorphic for $\Re(\beta) > 0$ and $\beta(j-s) \neq (2k+1)\pi i$. In which,

\[
F_\beta(s+1) = \frac{f(F_\beta(s))}{e^{-\beta s} + 1}\\
\]

In \cite{Nix5} we used these functions to construct an inverse Abel function for $f(z) = e^z$. In the general case, it may be difficult to do this unilaterally. But, we can still describe that this solves the asymptotic Abel equation; given our function $f$ has certain growth conditions.

\begin{theorem}
Suppose $|f(u)|/|f'(u)| < b$ for large $u$ and some $b \in \mathbb{R}^+$. The function $F_\beta$ satisfies the asymptotic Abel equation,

\[
f^{-1}(F_\beta(s+1)) - F_\beta(s) = \mathcal{O}(e^{-\beta s})\,\,\text{as}\,\,|s|\to\infty\,\,\text{while}\,\, \Re(\beta s) > 0\\
\]

wherever $F_\beta(s) \to \infty$ and $f(F_\beta(s)) \to \infty$.
\end{theorem}

\begin{proof}
Since $F_\beta(s+1) = \frac{f(F_\beta(s))}{e^{-\beta s}+1}$; we need only a theorem that,

\[
f^{-1}(\frac{f(u)}{e^{-\beta s}  +1}) - u = \mathcal{O}(e^{-\beta s})
\]

Using a Taylor expansion, we have an expression for the error term,

\[
\mathcal{O}(\frac{f(u)e^{-\beta s}}{f'(u)})\\
\]

By our assumption, we know that $|f(u)|/|f'(u)| < b$. This concludes the proof.
\end{proof}

This theorem is very beneficial if we know that $F_\beta \to \infty$ and $f(\infty) \to \infty$, in some way or manner. In \cite{Nix5}, we discussed $f(z) = e^z$, in which the Julia set is all of $\mathbb{C}$; and so we should expect its orbits to be unbounded. If we pick well enough, or work on the right paths, than iterates will be unbounded. With the case of $f(z) = e^z$ the function $f(F_\beta(s))/F_\beta(s+1) = 1 + \mathcal{O}(e^{-\beta s})$, can be interpreted in the right half plane of $s$ (Just like an Abel equation); with a fixed point at $\infty$.

In general circumstances--this isn't as easy. But if we think of a set $s \in \mathcal{A}$, such $s \to \infty$; and $f(s) \in \mathcal{A}$. Then, the implicit solution $F_\beta(s) + v(s)$ to the equation $F_\beta(s) + v(s) = f(F_\beta(s-1) + v(s-1)) + \mathcal{O}(e^{-\beta s})$, is solvable; given we have $|f(u)|/|f'(u)| < b$ for some $b \in \mathbb{R}^+$.

We can develop a series $v_n$ by continuing this idea,

\[
f^{\circ -n}(F_\beta(s+n)) = F_\beta(s) + v_n(s)\\
\]

Which, follows inductively on $n$. Upon which, it means that the orbits of $f$ eventually look like translations in $F_\beta$--at least for large $|s|$; and how close they look to orbits can be controlled by $v_n$. This is such a strong asymptotic relationship; repeated derivatives will also abide by these rules. Where, now trying to construct an Abel function $f(A(z)) = A(z+1)$--we can at least expect a similar asymptotic relationship.

In \cite{Nix5} this was emphasized with our solution $\tet$ of $f(z)=e^z$; which was constructed from this asymptotic relationship. The author isn't sure of what functions $f$ this method necessarily generalizes to; but the general limit formula is $f^{\circ -n}F(s+n)$ for an appropriate $F(s)$ depending on what kind of behaviour we want.\\

And from this point, we can change back into the local coordinate. We write,

\[
P_\lambda(w) = \OmSum_{j=1}^\infty \frac{wf(z)}{\lambda^{-j} + w}\,\bullet z\\
\]

This expression is satisfied for $w \neq -\lambda^{-j}$; and thus $P_\lambda(w)$ is holomorphic for $\mathbb{C}^\times = \{w \in \mathbb{C}\,|\, w \neq -\lambda^{-j}\}$. The function $P_\lambda$ satisfies the \textit{almost}-Schr\"{o}der equation:

\[
P_\lambda(\frac{w}{\lambda}) = \frac{w}{w+1}f(P_\lambda(w))\\
\]

Which exists as the change of variables, $\lambda = e^{-\beta}$ and $\log(w)/\beta = s$ and $P_\lambda(w) = F_\beta(s)$. The existence and proof of holomorphy of this thing, is handled by The Modified Schr\"{o}der Theorem \ref{thmSCH} or The Modified Abel Theorem \ref{thmABL}. We can write this asymptotic in the local scenario;

\[
f^{-1}(P_\lambda(\frac{w}{\lambda}))-P_\lambda(w) =  \mathcal{O}(w^{-1})\\
\]

Such is such, that now, we get a much nicer theorem:

\begin{theorem}[The Good Asymptotics Theorem]\label{thmGA}
Assume $|f(u)|/|f'(u)| < b$ to large $u$ and some $b \in \mathbb{R}^+$. The function $P_\lambda(w) = F_\beta(s)$ satisfies the asymptotics,

\[
f^{-1}(P_\lambda(\frac{w}{\lambda}))-P_\lambda(w) =  \mathcal{O}(w^{-1})\\
\]

for all $P_\lambda(w/\lambda) \to \infty$ and $f(P_\lambda(w)) \to \infty$.
\end{theorem}

The author would like the reader to visualize this theorem at its fullest. For this reason, call $\Psi^{-1}$ the inverse Schr\"{o}der function of $f$. This means that $f(\Psi^{-1}(z)) = \Psi(Lz)$. Then, choose a path $w \in \gamma$, such that $Lw \in \gamma$, which tends to $\infty$ in which,

\[
\Psi^{-1}(w) \to \infty
\]

Then,

\[
f^{-1}(\Psi^{-1}(Lw)) - \Psi^{-1}(w) = 0\\
\]

Which is just Schr\"{o}der's equation. The Good Asymptotic Theorem \ref{thmGA}, is a way of having Schr\"{o}der's equation at infinity. Alas, as we approach infinity in a certain manner, we can get closer and closer to solving Schr\"{o}der's equation; but we have a variation in the multiplier $\lambda$.

Nonetheless, from this viewpoint, we have a good view of the asymptotic behaviour of the orbit $f^{\circ n}$. In that, the error drops off well enough for $f^{\circ -n}(P_\lambda(\frac{w}{\lambda^n}))$ to be normal. If this sequence of functions is a meaningful thing, the limit is discoverable.

Now in this scenario, we chose the rational function $p(w) = \frac{w}{1+w}$. If we choose different rational functions, we can describe far more complex relationships at $\infty$; and at alternative fixed points. The function $p(w) = \frac{w}{1+w}$ is valuable because it satisfies $p(\infty) = 1$ on the Riemann sphere. Upon which, we enter into the discussion of dynamics on the Riemann sphere.

\section{The Riemann sphere and a not so familiar view of infinite compositions.}

For this section $f : \mathbb{C} \to \mathbb{C}$, but $p: \widehat{\mathbb{C}} \to \widehat{\mathbb{C}}$ while $p$ does not necessarily fix $0$. Instead, we'll say that $p(\infty) = 0$. Then, we describe the function,

\[
P_\lambda(w) = \OmSum_{j=1}^\infty p(\lambda^{-j} w)f(z)\,\bullet z\\
\]

The purpose of this section is to say, that since $p$ is a rational function; it is well-behaved at infinity. \`{A} la, car, $f(\infty) = \infty$ some where (by Picard, Liouville), we can regulate $f(\infty)$ using $p(\infty)$. This means that, $P_\lambda$ is holomorphic everywhere $p(\lambda^{-j} w)$ doesn't have singularities in $w$ for each $j$.

The function $P_\lambda(w)$ is holomorphic on $\mathbb{C}$ minus a set of countable points. And for some domain $\mathbb{E}$ we have, $w
\in \mathbb{E} \,\Rightarrow\,\lambda w \in \mathbb{E}$; and $P_\lambda(w\lambda^n) \to \infty$ for $w \in \mathbb{E}$ and $n\to\infty$. Which is a consequence of $f$ being an entire function with some point of divergence.

This serves to describe a coordinate $w$ such that $w+\mathcal{O}(\frac{1}{w})\mapsto z$. Where iterated maps $f^{\circ n}(z) \mapsto \lambda^{n}w$; at least at the point $w =0,\,z = \infty$. The value $\infty$ is, again, interpreted as unbounded. We don't refer to its growth explicitly. Just that there is growth.

In \cite{Nix5} we used a change of variables to describe the orbits of $f(z) = e^z$ by mapping the singularity to $0$. The manner we did this is to take $p(w) = \frac{1}{w + 1}$. Then the function of interest is,

\[
P_\lambda(w) = \OmSum_{j=1}^\infty \frac{f(z)}{\lambda^{-j} w + 1}\,\bullet z\\
\]

The point where $w=0$ is a point of divergence of the infinite composition. And $P_\lambda(w)$ is holomorphic on $w \neq -\lambda^j,0$. But how it diverges is the key. The map $P_\lambda(\lambda w) = \frac{f(P_\lambda(w))}{w+1}$ which at $0$, we can guess that $\infty = P_\lambda(0) = f(\infty)$. Or rather, that our function will diverge at $0$ how $f$ diverges at $\infty$.

This manner of viewing things, coupled with The Good Asymptotics Theorem \ref{thmGA}, gives us a way of constructing a Schr\"{o}der function in a neighborhood of $0$. We write this,

\[
f^{-1} P_\lambda(\lambda w) - P_\lambda(w) = \mathcal{O}(w)\\
\]

Which describes an asymptotic equation at $w=0$; so long as we take an appropriate path $w \to 0$. In this space, we are taking the contraction mapping $\mathcal{T} h(w) = h(\lambda w)$ and the inverse contraction, $h \mapsto f^{-1}(h(w))$. Having good control over this operation allows us to consider,

\[
f^{\circ -n} (P_\lambda(\lambda^n w)) - P_\lambda(w) = \mathcal{O}(w)\\
\]

Which, the iterate converges well enough--at least in the $f(z) = e^z$ scenario. This allows us to eventually solve the Schr\"{o}der equation; and it does so with no mention of a fixed point needed. It is constructed at the point $\infty$; where we make the odd claim that $f(\infty) = \infty$--where this is not interpreted in the usual manner of a point on the Riemann Sphere. But rather, interpreted as unbounded, non-normal convergence. Which is simply Picard's theorem at $\infty$; where either we approach infinity, or we oscillate between every value in the complex plane at $\infty$. 

In no manner is it necessary for there to be a fixed point at $\infty$ of $f$; simply that some growth condition exists. Which like, for $e^z$; its orbits are dense in $\mathbb{C}$; and are so unbounded; and well enough, the accumulation points for almost all $z \in \mathbb{C}$ of the orbit $e^z$ is the orbit at $0$--which diverges to infinity very fast.

So, if the function,

\[
v_n(w) = f^{\circ -n} P_\lambda(\lambda^n w) - P_\lambda(w)\\ 
\]

converges uniformly as $n\to\infty$; then the function,

\[
\Phi_\lambda(w) = P_\lambda(w) + v(w)\\
\]

Satisfies the Schr\"{o}der equation,

\[
\Phi_\lambda(\lambda w) = f(\Phi_\lambda(w))\\
\]

Where this equation is interpreted in some neighborhood of the point $w=0$; where $\Phi_\lambda(0) = \infty$; can be interpreted at least in some way. This derides a way of constructing a Schr\"{o}der function at $\infty$; again, ignoring much of the local theory of fixed points--when usually constructing Schr\"{o}der functions.

Now, as to the uniqueness of these things--there isn't much to say. This method is intended for very difficult functions like $e^z$; which betray much of the traditional theory. Where, when using $e^z$; and trying to construct a function $e^{A(z)} = A(z+1)$; the Schr\"{o}der method of iteration is undesirable. As, using the Schr\"{o}der function produces a non-real valued function--unless we massage it gratuitously. And the desired result is a real-valued $A(z)$. 

This decries any method of showing a quick uniqueness criterion. As it implies a uniqueness criterion on all real valued Abel functions of $e^z$. And the author cannot think of anything obvious to this. In fact, this is much of an open-theory; discussing a uniqueness of real-valued tetrations.

This method of constructing a Schr\"{o}der function, doesn't violate the uniqueness condition we see in Milnor; which is $\Psi(0) = z_0$ and $\Psi'(0) = 1$; as we do not consider a fixed point $f(z_0) = z_0$. But, because of this, we lose any chance of harmony; there may be many more solutions like the one $\Phi_\lambda$--though there is one for each $0 < |\lambda| < 1$.

With that, we state our culminated theorem on the construction of a Schr\"{o}der function--where everything is right.

\begin{theorem}[An Alternative Schr\"{o}der Function]\label{thmAltSCH}
Suppose $f:\mathbb{C} \to \mathbb{C}$ is an entire function. Assume for large $u \in \mathcal{A}$ that $|f(u)|/|f'(u)| < b$ for some $b \in \mathbb{R}^+$. Let $f(\mathcal{A}) \subset \mathcal{A}$ and $f^{\circ n}(\mathcal{A}) \to \infty$. Assume additionally that,

\[
|f^{-1}(u) - f^{-1}(u')| \le (1+\epsilon)|u-u'|\\
\]

Where $\epsilon \to 0$ as $u,u'\to\infty$ on $\mathcal{A}$. Then, there exists an inverse Schr\"{o}der function $\Phi_\lambda$ for $w \in \mathcal{U} = \{0 < |w| < \delta,\,w\neq-\lambda^{j}\}$ and $0 < |\lambda| < 1$, which sends to $\mathcal{A}$, in which,

\[
\Phi_\lambda(\lambda w) = f(\Phi_\lambda(w))\\
\]
\end{theorem}

\begin{proof}
Define the sequence of terms $v_n(w) = f^{\circ -n} P_\lambda(\lambda^n w) - P_{\lambda}(w)$; which are holomorphic for $|w| < \delta$ for $w\neq -\lambda^j$. We see this because the equation $f^{-1}(P_\lambda(\lambda w) + \mathcal{O}(w)) - P_\lambda(w) = \mathcal{O}(w)$--which removes the singularity of $v_n(w)$ at $w =0$. As this asymptotic is satisfied in a neighborhood of $0$. This can be written as the process,

\[
v_{n+1}(w) = f^{-1}(P_\lambda(\lambda w) + v_n(\lambda w)) - P_\lambda(w)\\
\]

Then, by inspection,

\[
|v_{n+1} - v_n| \le (1+\epsilon)|v_n(\lambda w) - v_{n-1}(\lambda w)|\\
\]

There exists $0 < |\lambda| < q < 1$ for each $\delta,\delta'>0$--where $q \to |\lambda|$ as $\delta \to 0$; in which,

\[
||v_{n+1} - v_n||_{\mathcal{B}} \le (1+\epsilon)q||v_n - v_{n-1}||_{\mathcal{B}}\\
\]

For a compact set $\mathcal{B} = \{w \in \mathbb{D}\,|\,|w|\le\delta,\,|w+\lambda^j| \ge \delta'|\lambda|^j\}$. Where, $(1+\epsilon)q < 1$ with an appropriate choice of $\delta>0$. By Banach's Fixed Point Theorem; we have uniform convergence in $\lambda$ and $w$; call $v_n \to v$. The function,

\[
\Phi_\lambda(w) = P_\lambda(w) + v(w)\\
\]

Satisfies,

\[
\Phi_\lambda(\lambda w) = f(\Phi_\lambda(w))\\
\]
\end{proof}

\section{In Conclusion}

We have walked the reader through the framework of our paper \cite{Nix5}. We did so focusing on more adaptable methods of the construction present there. Much of this work is reserved for fast growing entire functions; which have a reasonable divergence at $\infty$. The function $e^z$ is the prototypical example of such a function. But if $f$ diverges fast enough at infinity; such its derivative $f'(u)$ diverges in a faster or similar enough manner to derive $|f(u)|/|f'(u)| < b \in \mathbb{R}^+$--then much of the work is generalizable.

Again, we thank the reader for their time.

\end{document}